\documentclass [a4paper, 12pt]{amsart}

\usepackage[margin=2.2cm]{geometry} 
\usepackage{amsrefs}
\usepackage{times}
\usepackage{wasysym,stmaryrd,enumerate}
\usepackage{amssymb,amsfonts,amsrefs,amsthm,mathrsfs,amsmath,amscd,version,graphicx,bbm, xspace}
\usepackage{tikz-cd}
\usepackage{amsmath}
\usepackage{amsfonts}
\usepackage{amssymb}
\usepackage{amsthm}
\usepackage{graphicx}
\usepackage[all]{xy}
\usepackage[english]{babel}

\newtheorem{theorem}{Theorem}[section]
\newtheorem{corollary}[theorem]{Corollary}

\newtheorem{lemma}[theorem]{Lemma}

\newtheorem{theo}{Theorem}

\theoremstyle{definition}

\newtheorem*{claim}{Claim}

\newtheorem{definition}[theorem]{Definition}

 \newtheorem{remark}[theorem]{Remark}
	\newtheorem{que}[theorem]{Question}
\newtheorem*{question}{Question}	
\newtheorem*{defii}{Definition}
\newtheorem*{corol}{Corollary}
\newtheorem*{ttt}{Theorem}
\DefineSimpleKey{bib}{arx}
\BibSpec{arx}{%
  +{}{\PrintAuthors} {author}
  +{,}{\textit} {title}
  +{}{ \parenthesize} {date}
  +{,}{ } {note}
  +{.}{ } {transition}
}

\begin{document}

\title[F\o lner functions and the generic WP for finitely generated amenable groups]{F\o lner functions and the generic Word Problem for finitely generated amenable groups}

\author{Matteo Cavaleri}
\address[M. Cavaleri]{Universit\`{a} degli Studi Niccol\`{o} Cusano - Via Don Carlo Gnocchi, 3 00166 Roma, Italia}
\email{matte.cavaleri@gmail.com}

\begin{abstract}
We  introduce and investigate different definitions of effective amenability, in terms of computability of F\o lner sets, Reiter functions, and F\o lner functions. As a consequence, we prove that recursively presented amenable groups have subrecursive F\o lner function, answering a question of Gromov; for the same class of groups we prove that solvability of the Equality Problem on a generic set (generic EP) is equivalent to solvability of the Word Problem on the whole group (WP), thus providing the first examples of finitely presented groups with unsolvable generic EP. In particular, we prove that for finitely presented groups,  solvability of generic WP doesn't imply solvability of generic EP.

\end{abstract}

\maketitle

\section{Introduction}
In this paper we define and study some effective versions of amenability for finitely generated groups, in terms of computability of F\o lner sets, computability of Reiter functions and subrecursivity of F\o lner functions.

Let $\Gamma$ be a group generated by a finite subset $X$.
 Given $n\in\mathbb N$, we say (cf. \cite{VER}) that a non-empty finite subset $\Omega \subset \Gamma$ is an \emph{$n$-F\o lner set} (with respect to $X$) if
\begin{equation}\label{folner}\frac{|\Omega \setminus x \Omega|}{|\Omega|}\leq n^{-1}, \;\;\;\forall x\in X.\end{equation}
We denote by $\mathfrak F\o l_{\Gamma,X}(n)$ the set of all  $n$-F\o lner sets of $\Gamma$ with respect to $X$. Moreover, we say that 
a sequence $(\Omega_n)_{n\in \mathbb N}$ of subsets  of $\Gamma$ is a \emph{F\o lner sequence} if for every $n\in \mathbb N$, $\Omega_n\in \mathfrak F\o l_{\Gamma,X}(n).$
A related important notion is the F\o lner function $F_{\Gamma,X}$, introduced by Vershik  \cite{VER}, that measures the cardinality of the smallest F\o lner sets:
$$F_{\Gamma,X}(n):=\min\{|\Omega|:\;\; \Omega\in \mathfrak F\o l_{\Gamma,X}(n)\},$$
with the convention that $\min \emptyset := \infty$.
It is well known that the existence of a F\o lner sequence and the asymptotic behaviour of the function $F_{\Gamma,X}$ does not depend on the choice of $X$: we say that $\Gamma$ is \emph{amenable} if it admits a F\o lner sequence (and therefore  $F_{\Gamma,X}(n)<\infty,\;\forall n\in \mathbb N$).

A  function $f\colon \mathbb N\to \mathbb N$ is said to be \emph{recursive} if there exists an algorithm (Turing machine) that:\\
(i) stops for every input $n$;\\
(ii) \emph{ computes $f$}, that is, gives $f(n)$ as an  output.\\
 A function is \emph{subrecursive} if it admits a recursive upper bound. We refer to \cite{Mal} for general computability theory.

Vershik himself was interested in algorithmic behaviour of F\o lner functions, conjecturing the existence of arbitrarily fast  growing F\o lner functions. 
This was confirmed by Erschler \cite{A2},  who provided examples of finitely generated groups with F\o lner function growing faster than any given function, even non-subrecursive. In particular, the F\o lner sets of those groups are missing any algorithmic description. Analogous results were recovered in \cite{G,OO}. We finally mention that, most recently,  Brieussel and Zheng \cite[Cor 4.7]{Zheng} have shown that any non-decreasing function is asymptotically equivalent to the F\o lner function of some finitely generated group.

However the behaviour for finitely presented groups remained open:
\begin{question}{\cite[p.578, Gromov]{G}}
``(d) Is there an universal bound on the asymptotic growth of the F\o lner
functions of finitely presented amenable groups by a recursive (primitively recursive?)
function? (Maybe there is such a bound in every given recursive
class of presentations?). Or, at another extreme, are there finitely presented
amenable groups with so fast growing F\o lner function, such that their amenability is unprovable
in Arithmetic? (An enticing possibility would be this situation for the Thompson group).''
\end{question}
The above-mentioned possibility about Thompson group was studied in \cite{Moo}: if the Thompson group $F$ is amenable then its F\o lner function grows faster than any iterated exponential. For recursively presented groups, in  \cite{A1} Erschler showed that the asymptotics of the F\o lner function of the $k$-iterated wreath-product of $\mathbb Z$ is the $k$-th tetration of $n$.

One of our main results (Section \ref{reisec}) is the following partial answer to the aforementioned question of Gromov:
\begin{theo}\label{A}
The F\o lner function of a recursively presented amenable group is subrecursive.  Moreover, every recursively enumerable class of recursive amenable presentations admits a uniform recursive upper bound for the asymptotic growth of the corresponding F\o lner functions.
\end{theo}
\begin{proof}
The first sentence follows from Theorem \ref{imp}, the second from Corollary \ref{c3}.
\end{proof}

The main tool used in the proof of the above theorem is the construction of a uniform algorithm $\widehat{\mathfrak{K}}$, described in Theorem \nolinebreak \ref{imp}, that for any $n\in \mathbb N$ and any recursive presentation, provides, if it exists, a function on the associated free group whose pushforward on the group is $n$-invariant (an equivalent notion for amenability, see Section \ref{prel}). Let us fix some notation.

 With any finite set $X$ of generators of $\Gamma$, we associate a set $\mathsf X$ and a bijection $\varphi\colon \mathsf X\to X$. 
 We denote by $\mathbb F_{\mathsf X}$ the free group generated by  $\mathsf X$, and
 by $\pi_\Gamma\colon\mathbb F_{\mathsf X}\rightarrow\Gamma$ the unique epimorphism extending~$\varphi$. The group
$\Gamma$ has \emph{solvable Word Problem} (WP) if there exists an algorithm that for every  $\omega \in \mathbb F_{\mathsf X}$ as an input, stops and  establishes whether or not $\omega$ represents the identity in $\Gamma$ (i.e.\ $\pi_\Gamma(\omega)=1_\Gamma$). This is equivalent to saying that   $\ker \pi_\Gamma\subset \mathbb F_{\mathsf X} $ is \emph{recursive}.  We also say that $\Gamma$ is recursively (resp.\ finitely) presentable if there exists $R\subset \mathbb F_{\mathsf X}$ recursive (resp.\ finite), such that the normal closure   $R^{\mathbb F_{\mathsf X}}=\ker \pi_\Gamma$. Dehn in \cite{Dehn} first formulated the Word Problem, several years before the study about computability started. Only in the 1950s  \cite{Boone,Novikov} examples of finitely presented groups with unsolvable WP appeared. 

From a practical point of view, often in computer science it is not important the behaviour of an algorithm for the totality of the inputs, because it is possible that it is strongly influenced by a small, negligible, subset of inputs. Sometimes it is more interesting  to study the \emph{average} or the behaviour for \emph{most} of the inputs. 
This concept was developed even in group theory \cite{G1,G2,G3,AO}: we refer to   \cite{KMSS}  for an extensive discussion on the subject. In particular,  Kapovich, Miasnikov, Schupp and Shpilrain formally defined the concept of \emph{generic computability} and \emph{generic-case complexity}, especially focusing on algorithmic problems for finitely generated groups. We now present the \emph{ generic Equality Problem}.

Following \cite{Afinite}, we say that the Equality Problem (EP) is solvable on a subset $S\subset \mathbb F_{\mathsf X}$ if there exists an algorithm with input $(\omega_1, \omega_2) \in \mathbb F_{\mathsf X}\times \mathbb F_{\mathsf X}$, such that whenever $(\omega_1, \omega_2) \in S\times S$ the algorithm stops, establishing whether $\pi_\Gamma(\omega_1)=\pi_\Gamma(\omega_2)$ or not.
Notice that when $S$ is a subgroup, EP  is equivalent to the Word Problem for $S$.

Denoting by $B_n$ the ball of radius $n$ in $\mathbb F_{\mathsf X}$, a subset $S\subset  \mathbb F_{\mathsf X}$ is called \emph{generic} if \begin{equation}\label{ggg} \lim_{n\to\infty} \frac{|S\cap B_n|}{|B_n|}=1 ;\end{equation} a subset is \emph{negligible} if its complement is generic.
\begin{defii}
The group $\Gamma$ has \emph{solvable generic EP} if there exist a finite set of generators $X$ and a generic subset $S\subset \mathbb F_{\mathsf X}$ such that the EP is solvable on $S$.
\end{defii}

The dependence on the choice of the generating set $X\subset \Gamma$ in the above definition is, to our knowledge, presently unknown. Passing from classical computability problems to their generic version fails, in general, to preserve independence of the choice of the generating set. However, we believe that, in the present setting, this is not the case.

The transition to genericity makes solvable some classical unsolvable problems; the literature in this direction is very rich, starting from \cite{KMSS} to \cite{JS,DJS,KS,KKS}. But, not less important, especially for cryptography, is to produce examples \cite{Afinite,GMO,MR,JS} of problems generically hard or even generically undecidable. Up to now there were no examples of finitely presented groups with unsolvable generic  WP or unsolvable generic  EP. Here we provide examples of the latter by proving a sort of ``stability" for the Word Problem in recursively presented amenable groups:
\begin{theo}\label{B}
In the class of recursively presented amenable groups:
$$\mbox{ solvable WP } \Longleftrightarrow \mbox{solvable generic EP}$$
\end{theo}
\begin{proof}Section \ref{GEP} is devoted to proving this theorem.
\end{proof}
To prove this, we use a variation of the algorithm $\widehat{\mathfrak K}$, and the following: in a recursively presented group computability, for every $n$, of a one-to-one preimage of an $n$-F\o lner set,  gives solvability of WP (Theorem \ref{WP}). Thus, more generally, solvability of EP on a set containing a preimage of a F\o lner sequence implies solvability of the WP.

As a byproduct, the following provides a solution to \cite[Problem 1.5, b]{Afinite} (we denote by $G(M)$ the Kharlampovich groups, see \cite{H,KMS}):
\begin{corol}
The finitely presented groups $G(M)$ have  unsolvable generic Equality Problem.
\end{corol} 
Indeed, the groups $G(M)$ are finitely presented, solvable and therefore amenable, and have unsolvable Word Problem (\cite{H,KMS}).

Note that in \cite{KMSS} (linear) solvability of the generic Word Problem for solvable groups  is proved. Thus even if the Equality Problem is the natural generalization of the Word Problem, however the generic EP is different from the generic WP.

Let $\mathcal C_A$ denote the class of recursively presented amenable groups and consider the following  subclasses: $\mathcal C_{WP}$ (with solvable WP), $\mathcal C_{CF}$ (with computable F\o lner sets), $\mathcal C_{CFI}$ (with computable F\o lner sets by one-to-one preimages), $\mathcal C_{CR}$ (with computable Reiter functions), $\mathcal C_{SF}$ (with subrecursive F\o lner function) (see next section for the definitions). \\The following theorem summarizes the current understanding about the relations among these se\-veral notions of effective amenability.
\begin{theo}\label{C}
$$\mathcal C_{CFI}=\mathcal C_{WP}\subsetneq \mathcal C_{CF} \subset \mathcal C_{SF}=\mathcal C_{CR}=\mathcal C_A$$\end{theo}
\begin{proof}
The first equality is Theorem \ref{WP}, the other equalities follow from Theorem \ref{imp}, the re\-maining relations were already proved in \cite{CAV,Pre}.
\end{proof}
\noindent Whether or not the inclusion $\mathcal C_{CF} \subset \mathcal C_{SF}$ is strict is an open question.

\vspace{16mm}

\begin{center}
The paper is organized as follows.
\begin{itemize}
\item[Section \ref{prel}]
We introduce notation and the definitions of  \emph{computable F\o lner sets} and \emph{computable Reiter functions}. We present some basic properties, fundamental for all the next sections.
\item[Section \ref{reisec}]
We prove that every amenable recursively presented group has computable Reiter functions and therefore has subrecursive F\o lner function, equivalently, $\mathcal C_{SF}=\mathcal C_{CR}=\mathcal C_A$ (Theorem~\ref{imp}). We analyze the existence of uniform recursive upper bounds for the F\o lner functions of  recursively presented amenable groups (Corollary \ref{c1}, \ref{c2}, \ref{c3}). Theorem \ref{imp} and its proof are fundamental for  Section~\ref{GEP}.
\item[Section \ref{swp}]
In the class of amenable recursively presented groups, we characterize those groups with solvable WP as the groups with computable F\o lner sets by one-to-one preimages, equivalently, we show $\mathcal C_{CFI}=\mathcal C_{WP}$ (Theorem \ref{WP}). Moreover, in Corollary \ref{ovvio}, we easily show that, in case of solvability of WP, all definitions of effective amenability are equivalent. Theorem~\ref{WP} is fundamental for Section \ref{GEP}.
\item[Section \ref{GEP}]
We prove that a recursively presented amenable group with solvable generic EP has sol\-va\-ble WP (Theorem \ref{uno}). The proof uses the algorithm described in the proof of Theorem \ref{imp} and the characterization of the WP given by Theorem \ref{WP}.
\item[Section \ref{QR}]
Questions and final remarks.
\end{itemize}
\end{center}

\subsection*{Acknowledgement}
This work was  supported by a grant of the Romanian National Authority for Scientific Research and
Innovation, CNCS - UEFISCDI, project number PN-II-RU-TE-2014-4-0669. I thank Tullio Ceccherini-Silberstein for the long and precious discussions, and the anonymous referees: a first one who suggested me to investigate the notion of computability of Reiter functions and to strengthen the notion of computability of F\o lner sets (this suggestion turned very precious for my further development in computability theory), and a second one for the careful reading and the precious advices on the organization of the presentation of the paper.

\section{Preliminaries}\label{prel}
Throughout this paper, $\Gamma$ is a group generated by a finite set $X$. We fix a set $\mathsf X$ and a bijection $\varphi\colon \mathsf X\to X$, and denote by  $\pi_\Gamma\colon\mathbb F_{\mathsf X}\rightarrow\Gamma$ the unique epimorphism extending $\varphi$, where $\mathbb F_{\mathsf X}$ is the free group based on $\mathsf X$. For $x\in X$ we set $\mathsf x:=\varphi^{-1}(x)\in \mathsf X$: we believe that this use of different fonts, avoiding possible ambiguities, considerably simplifies notation. Given an element $\omega$ in the free group $\mathbb F_{\mathsf X}$ we denote by $|\omega|$ the natural word length of $\omega$ with respect to  $\mathsf X \cup \mathsf X^{-1}$; we denote by  $B_n:=\{\omega\in\mathbb F_{\mathsf X}:\; |\omega|\leq n\}$ the ball of radius $n$ and by ${S_n:= B_n\setminus B_{n-1}\subset\mathbb F_{\mathsf X},}$ the sphere of radius $n$.
 For a natural number $k$, we denote by ${[k]:=\{1,2,\ldots,k\}}$, and recall that $\mathfrak F\o l_{\Gamma,X}(n)$ is the family of $n$-F\o lner sets of $\Gamma$ with respect to $X$. The function $\chi_A$ is the characteristic function of the subset $A$ (both for $A\subset \Gamma$ or $A\subset \mathbb F_{\mathsf X}$).

\begin{definition}
A summable non-zero function $h\colon \Gamma \to \mathbb R^+$, $\|h\|_{1,\Gamma}:=\sum_{g\in\Gamma} |h(g)|<\infty $, is \emph{$n$-invariant} with respect to $X$ if for all $x\in X$ 
\begin{equation}\label{nin}\frac{\|h-_x\!h\|_{1,\Gamma}}{\|h\|_{1,\Gamma}}\leq n^{-1};\end{equation}
 where  $_xh \colon \Gamma \to \mathbb R^+$ is the function defined by $_xh(g):=h(x^{-1}g)$.\\
We denote by $\mathfrak Reit_{\Gamma,X}(n)$ (from the \emph{Reiter condition} for amenability \cite{rei}) the set of all summable non-zero  functions from $\Gamma$ to $\mathbb R^+$ that are $n$-invariant  with respect to $X$.
\end{definition}
\begin{remark}\label{RF} The following facts are well known and/or easy to prove (see \cite{tullio,Coo})
\begin{itemize}
\item
$\Omega \in \mathfrak F\o l_{\Gamma,X}(n) \implies  \Omega g \in \mathfrak F\o l_{\Gamma,X}(n), \; \forall g\in \Gamma;$
\item
$\Omega \in \mathfrak F\o l_{\Gamma,X}(n) \implies \frac{|\Omega\setminus x^{-1} \Omega|}{|\Omega|}\leq \frac{1}{n},\; \forall x\in X;$
\item
$\Omega\in \mathfrak F\o l_{\Gamma,X}(n) \Leftrightarrow \frac{|\Omega\cap x \Omega|}{|\Omega|}\geq 1-\frac{1}{n},\; \forall x\in X;$

\item
$\Omega\in \mathfrak F\o l_{\Gamma,X}(2n)\: \Leftrightarrow \chi_\Omega\in \mathfrak Reit_{\Gamma,X}(n),$\\
since $\frac{\|\chi_\Omega-_x\!\chi_\Omega\|_{1,\Gamma}}{\|\chi_\Omega\|_{1,\Gamma}}= 
\frac{\|\chi_\Omega-\!\chi_{x\Omega}\|_{1,\Gamma}}{\|\chi_\Omega\|_{1,\Gamma}}=2 \frac{|\Omega\setminus x \Omega|}{|\Omega|};$

\item
$h\in \mathfrak Reit_{\Gamma,X}(n) \implies \exists \Omega\subset Supp(h):=\{ g\in \Gamma:\; h(g)\neq 0\},\; \Omega\in \mathfrak F\o l_{\Gamma,X}(n)$,\\
precisely, by the so-called layer cake decomposition, or Namioka's trick, there exists ${\epsilon\in \mathbb R^+}$ such that $\{g\in \Gamma:\; h(g)>\epsilon\}\in \mathfrak F\o l_{\Gamma,X}(n);$
\end{itemize}

\end{remark}
Thus $\Gamma$ is amenable if and only if $\mathfrak Reit_{\Gamma,X}(n)\neq \emptyset$ for every $n\in\mathbb N$  or, equivalently, ${\mathfrak F\o l_{\Gamma,X}(n)\neq \emptyset}$ for every $n\in\mathbb N$. 
In order to define a notion of effective amenability for $\Gamma$ we require the existence of an algorithm computing, in some sense, either F\o lner sets or  Reiter functions.
Since in general $\Gamma$ has unsolvable Word Problem we ``lift" the output to $\mathbb F_{\mathsf X}$.
The following notion was introduced and studied in \cite{CAV,Pre}:
\begin{definition}
$\Gamma$ has \emph{computable F\o lner sets}  if there exists an algorithm with:\\
INPUT: $n\in \mathbb N$\\
OUTPUT: $F \subset \mathbb F_X$ finite, such that $\pi_{\Gamma}(F)\in\mathfrak F\o l_{\Gamma,X}(n)$.
\end{definition}
The computability of F\o lner sets does not depend on the choice of the  finite set of generators and, in particular, for finitely presented groups, if we change a given finite presentation we can algorithmically update the algorithm.\\
The following is the analogue definition for the Reiter condition:
\begin{definition}\label{rei}
$\Gamma$ has \emph{computable Reiter functions} with respect to $X$ if there exists an algorithm with
\\
INPUT: $n\in \mathbb N$\\
OUTPUT: $f\colon \mathbb F_{\mathsf X}\to \mathbb Q^+$, finitely supported, such that ${\pi_{\Gamma}}_*(f)\in \mathfrak Reit_{\Gamma,X}(n)$,
\\where ${\pi_{\Gamma}}_*(f)\colon \Gamma \to \mathbb Q^+$
 is the \emph{pushforward} of $f$, defined by ${\pi_{\Gamma}}_*(f)(g):=\sum_{\nu\in \pi_{\Gamma}^{-1}(g)}f(\nu)$. 
\end{definition}

\begin{remark}
Consider the commutative diagram of group epimorphisms:

\[
  \begin{tikzcd}
    G_1 \arrow{r}{\pi_1} \arrow[swap]{dr}{\pi_3} & G_2 \arrow{d}{\pi_2} \\ & G_3
  \end{tikzcd}
\]
and $f \colon G_1 \to \mathbb R.$ Then the following holds:
\begin{itemize}
\item
${\pi_2}_*({\pi_1}_*(f))={\pi_3}_*(f)$ and if $f$ is finitely supported then ${\pi_1}_*(f)\colon G_2\to \mathbb R$ is finitely supported; \\
as a consequence, the definition of computability of Reiter functions does not depend on the choice of the finite set of generators;
\item
${\pi_1}_*(_gf)=\,_{\pi_1(g)}\!\, {\pi_1}_*(f),\; \forall g\in G_1$;
\item
$\|f\|_{1,G_1}\geq \|{\pi_1}_*(f)\|_{1,G_2}\geq \|{\pi_3}_*(f)\|_{1,G_3}$, and, if $f$ is positive, equalities hold;
\item
 ${\pi_1}_*(f) \in \mathfrak Reit_{G_2}(n)\implies {\pi_3}_*(f)\in \mathfrak Reit_{G_3}(n)$,\\
thus computability of Reiter functions passes to  quotients.

\end{itemize}
\end{remark}

\section{Recursive bounds for F\o lner functions}\label{reisec}

\begin{theorem}\label{imp}
Suppose that $\Gamma$ is recursively presentable. Then the following are equivalent:
\begin{itemize}
\item[(i)]
$\Gamma$ is amenable;
\item[(ii)]
$\Gamma$ has subrecursive F\o lner function;
\item[(iii)]
there exists an algorithm with \\
INPUT: $n\in \mathbb N$\\
OUTPUT: $F \subset \mathbb F_{\mathsf X}$ finite, such that $\pi_{\Gamma}(F)$ contains an $n$-F\o lner set;
\item[(iv)]
$\Gamma$ has computable Reiter functions.
\end{itemize}
\end{theorem}
\begin{proof}
It is clear that $(iii)\implies (ii)\implies (i)$; \\
$(iv)\implies (iii)$ \\
For every $n\in \mathbb N$ the output of the algorithm  in Definition \ref{rei} is  a function $f\colon \mathbb F_{\mathsf X}\to \mathbb Q^+$ with finite support, say
$F \subset \mathbb F_{\mathsf X}$. 
Let $h:={\pi_{\Gamma}}_*(f)$
 be the pushforward of $f$, so that $h\in \mathfrak Reit_{\Gamma,X}(n)$. Then, as mentioned in Remark \ref{RF},
there exists $\epsilon\in\mathbb R^+$ such that 
$\Omega_\epsilon:=\{g\in \Gamma: h(g)>\epsilon\}\in \mathfrak F\o l_{\Gamma,X}(n)$. 
We complete by observing that $\Omega_\epsilon\subset \pi_\Gamma(F)$.\\
$(i)\implies (iv)$ \\
The first step is to write, fixing $n\in\mathbb N$, a subroutine  $\mathfrak K(n)$  that, taken a function $f\colon \mathbb F_{\mathsf X}\to \mathbb Q^+$ with finite support $F\subset \mathbb F_{\mathsf X}$, stops if ${\pi_{\Gamma}}_*(f)\in \mathfrak Reit_{\Gamma,X}(n)$.
In fact, even if we cannot compute the pushforward (because we have no assumptions on WP), we can estimate the $n$-invariance after the following arguments.

With every partition $\mathcal Q$ of the finite support $F$ we associate the positive rational numbers $$M^{\mathsf x}_{\mathcal Q}(f):=\frac{\sum_{V\in \mathcal Q}|\sum_{\nu\in V}(f(\nu)-f(\mathsf x^{-1}\nu))|}{\sum_{\nu\in F}f(\nu)},\;  \mathsf x\in \mathsf X.$$ Denoting by $\mathcal P$ the canonical partition of $F$ associated with $\pi_\Gamma$ ($\forall \nu_1,\nu_2\in F$ there exists $V\in \mathcal P$ such that $\nu_1,\nu_2\in V$ if and only if $\pi_\Gamma(\nu_1)=\pi_\Gamma(\nu_2))$, we have

\begin{equation}\label{est}
\frac{\|{\pi_{\Gamma}}_*(f)-_x\!\!{\pi_{\Gamma}}_*(f) \|_{1,\Gamma}}{\|{\pi_{\Gamma}}_*(f)\|_{1,\Gamma}}=M^{\mathsf x}_{\mathcal P}(f),\; \forall \mathsf x\in \mathsf X.
\end{equation}
By the triangle inequality, for any two partitions $\mathcal Q$  and $\mathcal Q'$ of $F$ if  ${\mathcal Q \leq  \mathcal Q'}$ then $M^{\mathsf x}_{\mathcal Q}(f)\geq M^{\mathsf x}_{\mathcal Q'}(f)$.
In particular for any partition $\mathcal P'$ of $F$ such that $\mathcal P'\leq \mathcal P$, or equivalently, such that ${\nu_1,\nu_2\in V\in \mathcal P'}\implies \pi_\Gamma(\nu_1)=\pi_\Gamma(\nu_2)$, using equation \eqref{est} we have

\begin{equation}\label{est1}
\frac{\|{\pi_{\Gamma}}_*(f)-_x{\pi_{\Gamma}}_*(f) \|_{1,\Gamma}}{\|{\pi_{\Gamma}}_*(f)\|_{1,\Gamma}}\leq M^{\mathsf x}_{\mathcal P'}(f),\; \forall \mathsf x\in \mathsf X.
\end{equation}
So we define  $\mathfrak K(n)$ as follows: with input $f$, it sets $\mathcal P_0:= \{\{f\}: f\in F\}$, the finest partition of $F$.
As $\Gamma$ is recursively presented,  there is a recursive enumeration $\eta_1, \eta_2,\ldots$ of the words in $\ker \pi_\Gamma$.
When $\mathfrak K(n)$ reads $\eta_m$, for every pair of  distinct $V_1, V_2\in \mathcal P_{m-1}$ such that $\eta_m\in V_1 V_2^{-1}$, 
 it merges $V_1$ and $V_2$, defining a new partition $\mathcal P_{m}$; then it computes $M^{\mathsf x}_{\mathcal P_m}(f)$ and, if $M^{\mathsf x}_{\mathcal P_m}(f)\leq n^{-1}$ for every $\mathsf x\in \mathsf X$, it stops, if not, it goes to the next trivial word $\eta_{m+1}$.

By construction $\mathcal P_m\leq \mathcal P$ and the inequality \eqref{est1} holds (with $\mathcal P'=\mathcal P_m$);
 thus, when $\mathfrak K(n)$ stops, $M^{\mathsf x}_{\mathcal P_m}(f)\leq n^{-1}$ for every $\mathsf x\in \mathsf X$, and therefore ${\pi_{\Gamma}}_*(f)$ is $n$-invariant.
Conversely, if ${\pi_{\Gamma}}_*(f)$ is $n$-invariant, at latest when $\mathcal P_m=\mathcal P$ we have $M^{\mathsf x}_{\mathcal P_m}(f)\leq n^{-1}$, for any $\mathsf x\in\mathsf X$, by equality \eqref{est}.

Now, using hypothesis (i), for every $n\in \mathbb N$ there exists a non-empty finite subset $F\in \mathbb F_{\mathsf X}$ such that  $\pi_{\Gamma}(F)\in \mathfrak F\o l_{\Gamma,X}(2n)$ and $|F|=|\pi_\Gamma(F)|$: the pushforward of the characteristic function $\chi_F$ of $F$ is the characteristic function $\chi_{\pi_\Gamma(F)}\in \mathfrak Reit_{\Gamma,X}(n)$, by Remark \ref{RF}. We list all finite subsets of $\mathbb F_{\mathsf X}$:  $F_1,F_2,\ldots$ (they are countably many) and  we  simultaneously run $\mathfrak K(n)$ on $\chi_{F_1}$,  $\chi_{F_2} \ldots$ until one of the subroutines stops, providing a function with $n$-invariant pushforward (the sought Reiter funtion).
\end{proof}

\begin{remark}
In general,  the algorithm $\mathfrak K(n)$ may stop also with a function $\chi_F$ whose pushforward is not a characteristic function in $\Gamma$. This obstruction to reach $n$-F\o lner sets cannot be avoided because if we could change  $\mathfrak K(n)$ in order to stop only when ${\pi_{\Gamma}}_*(\chi_F)$ is characteristic, this would imply that $\Gamma$ has solvable Word Problem (this is a consequence of Theorem \ref{WP} that we will see in the next section). This is in general impossible, even for finitely presented groups with subrecursive F\o lner function.

The question --whether we can obtain computability of F\o lner sets (i.e.\ of a preimage not ne\-ces\-sarily 1-1)  with a  similar algorithm-- remains open: actually, we can estimate better and better $|\pi_\Gamma(F)\setminus x\pi_\Gamma(F)|$ from above listing the elements in $\ker \pi_\Gamma$, but in this case the denominator $| \pi_\Gamma(F)| $ is not computable and, at least for a general set, it  is impossible to estimate from below its cardi\-na\-li\-ty without solvability of the Word Problem. The same issue appears for stability of computability of F\o lner sets under quotients, see \cite{Pre}.
\end{remark}
Consider an enumeration $(P_i)_{i\in\mathbb N}$ of all finitely generated recursive presentations, $P_i=\{\mathsf  X_i| \mathsf R_i\}$, $\Gamma_i:=\mathbb  F_{\mathsf X_i}\slash  \mathsf R_i^{\mathbb F_{\mathsf X_i}}$. Clearly, we can extend $\mathfrak K$ to the universal algorithm $\widehat{\mathfrak K}$, that taking as an input $n$ and a presentation $P_i$, runs as $\mathfrak K(n)$ on $\mathbb  F_{\mathsf X_i}$, using only the recursive set of relations $\mathsf R_i$, and stops if the group $\Gamma_i$ admits $n$-F\o lner sets with respect to $X_i$.

Recall (cf. \cite{Mal}) that a \emph{partially recursive, $k$-place function} is a function $\mathcal U\colon D_{\mathcal U}\to \mathbb N$, where $D_{\mathcal U}\subset\mathbb N^k$, such that there exists an algorithm that for every input $(n_1,n_2,\ldots,n_k)\in  D_{\mathcal U}$ stops and gives $\mathcal U(n_1,n_2,\ldots,n_k)$ as an  output.

\begin{corollary}\label{c1}
There exists a $2$-place  partial recursive function $\mathcal U$  such that $$F_{\Gamma_i,\mathsf X_i}(n)\leq \mathcal U(i,n)$$ on the domain $\{(i,n)\in\mathbb N^2:\;F_{\Gamma_i,\mathsf X_i}(n)< \infty \}.$
\end{corollary}
\begin{corollary}\label{c2}
For every $n\in \mathbb N$ fixed, the set of finitely generated  recursive presentations of groups admitting $n$-F\o lner sets is recursively enumerable.
\end{corollary}
\begin{remark}
For every $n\in \mathbb N$ fixed the property of admitting $n$-F\o lner sets is a presentation property, not a group property.
\end{remark}
\begin{corollary}\label{c3}
For every recursively enumerable class $\mathcal C$ of finitely generated  recursive presentations of amenable groups there exists a recursive function $U_{\mathcal C}$ such that for every $P_i\in \mathcal C$:
$$F_{\Gamma_i,\mathsf X_i}\leq U_{\mathcal C} \;\; \mbox{eventually}.$$
\end{corollary}
\begin{proof}
More generally, suppose that $(f_i)_{i\in\mathbb N}$ is a recursively enumerable set of recursive functions $f_i\colon \mathbb N\to \mathbb N$. Then the function $U\colon \mathbb N\to \mathbb N$, defined as
$$U(n):= \max_{i\leq n} f_i(n)$$
is recursive and eventually dominates $f_i$, for every $i\in\mathbb N$.
\end{proof}
\noindent This concludes the proof of Theorem \ref{A} in the Introduction.

\section{Amenability and the Word Problem}\label{swp}

\begin{theorem}\label{WP}
The following are equivalent:
\begin{itemize}
\item[(i)]
$\Gamma$ is  amenable with solvable Word Problem;
\item[(ii)]
$\Gamma$ is recursively presentable and
there exists an algorithm with \\
INPUT: $n\in \mathbb N$\\
OUTPUT: $F \subset \mathbb F_{\mathsf X}\,$ finite, such that $\pi_{\Gamma}(F)\in \mathfrak F\o l_{\Gamma,X}(n)$ and $|F|=|\pi_\Gamma(F)|.$ 
\end{itemize}
\end{theorem}
\begin{proof}${}$\\$(i)\implies (ii)$\\
Suppose $\Gamma$ is amenable with solvable Word Problem. Then, by the latter property, for any given finite subset $F \subset \mathbb F_X$ we can algorithmically check if $\pi_{\Gamma}(F)\in \mathfrak F\o l_{\Gamma,X}(n)$ and $|F|=|\pi_\Gamma(F)|.$
Fixing an enumeration of the finite subsets of $\mathbb F_X$, we check these conditions until we find a suitable $F$, whose existence is guaranteed by  amenability of $\Gamma$.\\
Finally, solvability of the Word Problem ensures existence of a recursive set $R:=\ker \pi_\Gamma$ of defining relations of $\Gamma$.\\
$(ii)\implies (i)$\\
It is clear that $(ii)$ implies amenability of $\Gamma$. It remains to show that $\Gamma$ has solvable Word Problem.
By virtue of Remark \ref{RF}, we have that  $\mathfrak F\o l_{\Gamma,X\cup X^{-1}}(n)=\mathfrak F\o l_{\Gamma,X}(n)$. Moreover, solvability of the Word Problem does not depend on the choice of the generating set. We can therefore assume, without loss of generality, that $X=X^{-1}$.
For a given $\omega\in \mathbb F_{\mathsf X}$, we denote by $n:=\max\{|\omega|,3\}$ and compute 
a finite subset $F$  of $\mathbb F_{\mathsf X}$ such that 
$\pi_{\Gamma}(F)\in \mathfrak F\o l_{\Gamma,X}(n^2)$ and
 $|F|=|\pi_\Gamma(F)|=:k$. 
 We write $F=:\{f_1,f_2,\ldots, f_k\}$ and ${\mathsf X=:\{ \mathsf x_1, \mathsf x_2,\ldots, \mathsf x_d\}}$.
  We are going to algorithmically construct $d$ permutations $\sigma_1,\ldots,\sigma_d\in Sym(k)$ that are ``approximations"  for the left action of $x_1,\ldots x_d$ on $\pi_\Gamma(F)$,  interpreting $[k]$ as a copy of $\pi_\Gamma(F)$.
We have no assumptions on the Word Problem but the group $\Gamma$ is recursively presented,  thus $\ker \pi_\Gamma$ is recursively enumerable: in order to obtain the sought permutations we list the trivial words $\eta_1,\eta_2,\ldots, \eta_t,\ldots$ and then, for every $\ell\in [d]$, we construct, in a way that we will describe soon, a sequence of approximations $$\Sigma^0_\ell\subset \Sigma^1_\ell\subset \cdots \subset \Sigma^t_\ell \subset\cdots,$$ where 
$\Sigma^t_\ell $, for $t=0,1,\ldots$, is not yet a permutation of $[k]$ but just a subset of $[k]^2$, with the following property:
\begin{equation}\label{leftaction}
(i,j)\in \Sigma^t_\ell \implies \pi_\Gamma(\mathsf x_\ell f_i)=\pi_\Gamma(f_j).
\end{equation}
We start by setting $\Sigma^0_\ell=\emptyset$ for $\ell=1,\ldots d$. So, for $t=0$, property \eqref{leftaction} trivially holds.\\ As we list the elements of $\ker \pi_\Gamma$, we update the
 $\Sigma^t_\ell$'s in this way:  we read $\eta_t \in \ker \pi_\Gamma$, for each $(i,j)\in[k]^2$ and each $\ell\in[d]$ such that $\mathsf x_\ell f_i f_j^{-1}= \eta_t$ in $\mathbb F_{\mathsf X}$, we set $\Sigma^t_\ell=\Sigma^{t-1}_\ell\cup \{(i,j)\}$. In this way, property \eqref{leftaction} is maintained for every $t$.\\
We stop when we meet $\hat t$ such that $\min_\ell |\Sigma^{\hat t}_\ell|>(1-\frac{1}{n^2})k$.
 We then simply write $\Sigma_\ell$ instead of $\Sigma^{\hat t}_\ell$. \\
Indeed, since $\pi_{\Gamma}(F)\in \mathfrak F\o l_{\Gamma,X}(n^2)$,  by Remark \ref{RF},
 we have that 

\begin{equation}\label{dannata}
\frac{|\{(i,j):\,\mathsf x_\ell f_i f_j^{-1}\in \ker \pi_\Gamma\}|}{k}\geq
\frac{|\pi_{\Gamma}(F)\cap x_\ell \pi_{\Gamma}( F)|}{|\pi_{\Gamma}(F)|}>1-\frac{1}{n^2}.
\end{equation}
This guarantees that our procedure will stop.
Injectivity of  $\pi_\Gamma$ on $F$ guarantees that if $(i,j),(i',j')\in \Sigma_\ell$ are distinct then $i\neq i'$ and $j\neq j'$.
Then for $\ell=1,\ldots,d$ we can algorithmically choose $\sigma_\ell\in Sym(k)$,
 a permutation of $[k]$ such that  $(i,j)\in \Sigma_\ell \implies \sigma_\ell(i)=j.$

\begin{claim}
The permutations $\sigma_1,\ldots,\sigma_d$  have the following property 
\begin{equation}\label{dogma}
\ell_H(\omega(\sigma_1,\ldots,\sigma_d))
\begin{cases}
\leq \frac{1}{n}, \; \mbox{ if }\omega \in B_n\cap  \ker \pi_\Gamma \\
\geq 1 - \frac{1}{n},\;\mbox{ if } \omega \in B_n \setminus  \ker \pi_\Gamma
\end{cases}
\end{equation}
where  for $\sigma\in Sym(k)$ the positive real number $\ell_H(\sigma):=\frac{|\{i\in[k]:\;\sigma(i)\neq i\}|}{k}$ is the \emph{normalized Hamming length} of $\sigma$. 
\end{claim}
\begin{proof}[Proof of the claim]
Suppose $\omega=\mathsf x_{l_n}\ldots\mathsf x_{l_2}\mathsf x_{l_1}$, where 
$l_z\in [d]$, 
 $z=1,2,\ldots,n$.  We define the subset
 $$I_{\omega}:=\{i_0\in [k]: \,
\exists i_1, i_2,
\ldots, i_n \in [k] : \,(i_{t-1},i_{t})
\in \Sigma_{l_z}, 
\forall z\in [n] \}.$$ 
 
Informally, $ I_{\omega}$ is the set of $i \in [k]$ for which we can compute $\omega(\sigma_1,\ldots,\sigma_d)(i)=\sigma_{l_n}\ldots\sigma_{l_2}\sigma_{l_1}(i)$ only looking at $\Sigma_1,\ldots,\Sigma_d$.
In particular, by property \eqref{leftaction} of the $\Sigma_\ell$'s, we have:
\begin{equation}\label{parola}
i\in I_\omega \implies \pi_\Gamma(\omega f_i)=\pi_\Gamma(f_{\omega(\sigma_1,\ldots,\sigma_d)(i)}).
\end{equation}
Setting  $N_{\ell}:=\{i\in[k]: (i,j)\notin \Sigma_\ell \; \forall j\in [k] \}$, we can also write $I_\omega=\{i_0\in [k]: \, \sigma_{l_{n'}}\ldots\sigma_{l_2}\sigma_{l_1}(i_0)\notin N_{l_{n'}},\, \forall n'\in[n] \}.$\\
In order to estimate the cardinality of $I_\omega$, we define $\phi \colon [k]\setminus I_\omega \hookrightarrow N_{l_n}\sqcup\ldots \sqcup N_{l_2}\sqcup N_{l_1}$,\\
$\phi(i):=( n',i')$ where $n'$ is the smallest number in $[n]$ such that $\sigma_{l_{n'}}\ldots\sigma_{l_2}\sigma_{l_1}(i)\in N_{l_{n'}}$, and $i':=\sigma_{l_{n'}}\ldots\sigma_{l_2}\sigma_{l_1}(i)$. 
By construction of $\Sigma_\ell$, $|N_\ell|\leq \frac{k}{n^2}$, combining with the fact that the map $\phi$ is injective, we have
\begin{equation}\label{acc}\begin{split}
&|[k]\setminus I_\omega|\leq \sum^n_{z=1} |N_{l_z}^{s_z}|\leq \frac{k}{n},\\
&|I_\omega|\geq (1-\frac{1}{n})k.\end{split}
\end{equation}

Suppose $\omega\in \ker \pi_\Gamma$. Then, for $i\in I_\omega$, property \eqref{parola} implies  $\pi_\Gamma(f_{\omega(\sigma_1,\ldots,\sigma_d)(i)})=\pi_\Gamma(f_{i})$. By injectivity of $\pi_\Gamma$ on $F$, $i$ is a fixed point of $\omega(\sigma_1,\ldots,\sigma_d)$; by virtue of estimate \eqref{acc}, we the have 
$\ell_H(\omega(\sigma_1,\ldots,\sigma_d))\leq \frac{|[k]\setminus I_\omega|}{k}\leq \frac{1}{n}.$

If $\omega\notin \ker \pi_\Gamma$, then again by property \eqref{parola} we have that, for $i\in I_\omega$, $\pi_\Gamma(f_{\omega(\sigma_1,\ldots,\sigma_d)(i)})\neq\pi_\Gamma(f_{i})$.
This means that  $I_\omega$ contains only non-fixed points and therefore, by virtue of estimate \eqref{acc},\\
$\ell_H(\omega(\sigma_1,\ldots,\sigma_d))\geq \frac{|I_\omega|}{k}\geq 1- \frac{1}{n}.$
This ends the proof of the claim.\qedhere
\end{proof}
We are now in position to complete the proof of the theorem. Since the number $\ell_H(\omega(\sigma_1,\ldots,\sigma_d))$ is computable, by property \eqref{dogma} we can algorithmically  determine whether $\omega$ belongs to  $\ker \pi_\Gamma$ or not. Thus $\Gamma$ has solvable Word Problem (in the terminology of \cite{CAV} we actually proved that $\Gamma$ has \emph{computable sofic approximations}, see Theorem 3.3.1 in \cite{CAV}).
\end{proof}

\noindent In combination with Theorem \ref{WP} and the results in \cite{Pre}, this proves  Theorem \ref{C} in the Introduction.

\begin{corollary}\label{ovvio}
Suppose that $\Gamma$ has solvable Word Problem. Then the following are equivalent:
\begin{itemize}
\item[(i)]
$\Gamma$ is amenable;
\item[(ii)]
there exists an algorithm with \\
INPUT: $n\in \mathbb N$\\
OUTPUT: $F \subset \mathbb F_X$ finite, such that $\pi_{\Gamma}(F)\in \mathfrak F\o l_{\Gamma,X}(n)$ and $|F|=|\pi_\Gamma(F)|;$ 
\item[(iii)]$\Gamma$ has computable F\o lner sets;
\item[(iv)] $\Gamma$ has computable Reiter functions;
\item[(v)] $\Gamma$ has subrecursive F\o lner function.

\end{itemize}
\end{corollary}
\begin{proof}
By virtue of Theorem \ref{WP} we have $(i)\implies (ii)$. It is obvious that $(ii)\implies (iii)\implies (i)$ and that $(ii)\implies (v)\implies (i)$; by Remark \ref{RF} we have $(iv)\implies (i)$.\\ Finally $(ii)\implies(iv)$ because if $F \subset \mathbb F_X$  is finite, such that $\pi_{\Gamma}(F)\in \mathfrak F\o l_{\Gamma,X}(2n)$ and ${|F|=|\pi_\Gamma(F)|}$ then the pushforward of the characteristic function $\chi_F$ of $F$ is the characteristic function $\chi_{\pi_\Gamma(F)}$ of  $\pi_\Gamma(F)$: this is $n$-invariant by Remark \ref{RF}.
\end{proof}

\section{Generic EP}\label{GEP}
\noindent This section is devoted to proving the following theorem (cf.\ Theorem \ref{B} in the Introduction).
\begin{theorem}\label{uno}
Suppose that $\Gamma$ is amenable and recursively presentable. Then the following are equivalent:
\begin{itemize}
\item[(i)]
$\Gamma$ has solvable Word Problem;
\item[(ii)]
$\Gamma$ has solvable generic Equality Problem.
\end{itemize}
\end{theorem}
As stated in the Introduction, the Kharlampovich groups $G(M)$  are finitely presented, solvable and therefore amenable, and have unsolvable Word Problem (see \cite{H,KMS}). Therefore, by the previous theorem, they have unsolvable generic Equality Problem, thus providing a solution to \cite[Problem 1.5, b]{Afinite}.

\noindent In order to prove Theorem \ref{uno}, we need some preliminary results.

\begin{lemma}\label{elenco}
Suppose $\Gamma$ has solvable Equality Problem on $S$, where  $S\subset \mathbb  F_{\mathsf X}$.
Then there exists a  family $\mathcal A$ of finite subsets of  $\mathbb  F_{\mathsf X}$, with the following properties:
\begin{enumerate}[(1-{$\mathcal A$})]
\item $\mathcal A$ is recursively enumerable;
\item ${\pi_\Gamma}_{|_A}$ is injective $\forall A\in \mathcal A$;
\item  $\forall S'\subset S, \mbox{ $S'$ finite, }\; \exists A\in \mathcal A$ such that $ \pi_\Gamma(A)= \pi_\Gamma(S')$.
\end{enumerate}
\end{lemma}
\begin{proof}
Let $\mathfrak A$ be the associated algorithm for the solvability of the Equality Problem. Recall that $\mathfrak A$  (at least) stops on $S\times S$. We can easily define an algorithm $\mathfrak A'$ with input $B$, any finite subset  of $\mathbb  F_{\mathsf X}$, that checks if any two words in $B$ represent the same elements in $\Gamma$, that is, it checks if
${\pi_\Gamma}_{|_B}$ is injective.
Clearly $\mathfrak A'$ stops at least for every finite $B\subset S$. Thus we enumerate all finite subsets of $\mathbb  F_{\mathsf X}$: $B_1, B_2,\ldots$ and we simultaneously (diagonally) run $\mathfrak A'$ on these sets, and give as an output only those subsets $B$ for which the two following conditions are met:
$\mathfrak A'$ stops and $\mathfrak A'$ has checked that ${\pi_\Gamma}_{|_B}$ is injective. Let $\mathcal A$ be the set of these outputs.
Propeties (1-{$\mathcal A$}) and (2-{$\mathcal A$}) hold by construction of $\mathcal A$. For any finite $S'\subset S$, for each element of $\pi_\Gamma(S')$ we choose only one representative word in $S'$, obtaining a subset $A\subset S'\subset S$ such that $\pi_\Gamma(A)= \pi_\Gamma(S')$ and ${\pi_\Gamma}_{|_A}$ is injective. Then $A\in \mathcal A$ and the property (3-{$\mathcal A$}) is proved.
\end{proof}

\begin{lemma}[Upper Banach genericity]\label{genefol}
Suppose that $S$ is a generic subset of  $\mathbb F_{\mathsf X}$. Then for every finite subset $F\subset  \mathbb F_{\mathsf X}$ there exists $y\in \mathbb F_{\mathsf X}$ such that $Fy\subset S$. 

\end{lemma}
\begin{proof}
Since for every finite set $F$ there exists $k\in \mathbb N$ such that $F\subset B_k$, without loss of generality we may reduce to the case $F=B_k$.
We denote by $N:=S^c$, the complement of $S$; so that, being $S$ generic, $N$ is \emph{negligible}, that is $\frac{|N\cap B_n|}{|B_n|}\to 0$.
We want to prove that there exists $y\in \mathbb F_{\mathsf X}$ such that $N\cap B_k y=\emptyset$. Recall that $S_n:=B_n\setminus B_{n-1}$ is the $n$-sphere in  $\mathbb F_{\mathsf X}$.

For every $m\in \mathbb N$ we have:
$$B_{m+2k}\supset \bigsqcup_{\omega\in S_m}
 B_k a_{\omega} \omega,$$
where, for every $\omega\in S_m$, the word $a_{\omega}$ is a suitable element of $S_k$  such that $|a_{\omega} \omega|=m+k.$ Let's check the disjointness of the union.
 For all distinct $ \omega,\, \omega'\in S_m$, since $|\omega \omega'^{-1}|\geq 2$ we have 
$|a_{\omega}
\omega 
\omega'^{-1}
 a_{\omega'}^{-1}|\geq 2k+2.$
 By the triangular inequality, $B_k a_{\omega}\omega$ and $B_k a_{\omega'}\omega'$ are disjoint.

Suppose, by contradiction, that $N\cap B_k y\neq \emptyset$ for every $y$, then we have
\begin{equation}\label{quasi}\frac{|B_n\cap N|}{|B_n|}\geq 
\frac{|S_{n-2k}|}{|B_n|}
 \rightarrow
\frac{2|\mathsf X|-2}{(2|\mathsf X|-1)^{2k+1}}.\end{equation}
If $|\mathsf X|\geq 2$, this is impossible since the set $N$ is negligible. \\
If $|\mathsf X|=1$, we notice that in $B_n$ there are approximately $\frac{n}{k}$ disjoint copies of $B_k$ and the limit in \eqref{quasi} equals $\frac{1}{k}$, providing again a contradiction.
\end{proof}
\begin{remark}
There are other notions of genericity: for instance, one may replace the balls $B_n$ by the spheres $S_n=B_n\setminus B_{n-1}$ in Equation \eqref{ggg}. It follows from Cesaro's theorem that any $(S_n)$-generic set is also  $(B_n)$-generic. As a consequence, Lemma \ref{genefol} remains true if we suppose that $S$ is $(S_n)$-generic. Moreover, upper Banach genericity is strictly weaker than genericity: fixing $\mathsf x\in \mathsf X$, for any function $f\colon \mathbb N \to \mathbb N$ the subset $T_f:=\bigcup_{n\in\mathbb N} B_n \mathsf x^{f(n)}$ clearly contains an increasing sequence of translated balls, but the asymptotic behavior of the ratio $\frac{|T_f\cap B_n|}{|B_n|}$ can be arbitrary (it depends on the growth of $f$).
\end{remark}
\begin{lemma}\label{GF} Suppose that $\Gamma$ is amenable and $S\subset \mathbb F_{\mathsf X}$ is generic. Then $\pi_\Gamma(S)$ contains a F\o lner sequence:
$\forall n\in\mathbb N\; \exists\, \Omega_n\in \mathfrak F\o l_{\Gamma,X}(n)  \mbox{ such that } \Omega_n\subset \pi_\Gamma(S).$
\end{lemma}
\begin{proof}
Since $\Gamma$ is amenable, for every $n\in\mathbb N$ there exists a finite subset $F_n\subset  \mathbb F_{\mathsf X}$  such that ${\pi_\Gamma(F_n)\in \mathfrak F\o l_{\Gamma,X}(n)}$. Since $S$ is generic, then by virtue of Lemma \ref{genefol} there exists ${y_n\in \mathbb F_{\mathsf X}}$ such that $F_n y_n\subset S$; by Remark \ref{RF}, the set $\Omega_n:=\pi_\Gamma(F_n y_n)\in \mathfrak F\o l_{\Gamma,X}(n)$.
\end{proof}

\begin{proof}[Proof of Theorem \ref{uno}]${}$\\
$(i)\implies (ii)$ is true in general.\\
$(ii)\implies (i)$\\
By virtue of Theorem \ref{WP}, it is enough to show the existence of a finite generating set $Y$ and an algorithm with:\\
INPUT: $n\in \mathbb N$\\
OUTPUT: $F \subset \mathbb F_{\mathsf Y}$ finite, such that $\pi_{\Gamma}(F)\in \mathfrak F\o l_{\Gamma,Y}(n)$ and $|F|=|\pi_\Gamma(F)|.$ \\
Since $\Gamma$ has solvable generic Equality Problem, there exists a set of generators, say $Y$, and a generic subset $S\subset \mathbb F_{\mathsf Y}$ with solvable EP.\\ Let $\mathcal A$ be the family given by Lemma \ref{elenco}. By property (1-${\mathcal A}$) we have a recursive enumeration of $\mathcal A$:  $E_1,E_2,\ldots$ . 
Thanks to property (3-{$\mathcal A$}), the family $\pi_\Gamma(\mathcal A):=\{\pi_\Gamma(E_1), \pi_\Gamma(E_2),\ldots\}$ contains $\{\pi_\Gamma (S'):\; S'\subset S,\; S'\mbox{ finite}\}$ and, by Lemma \ref{GF}, for every $n \in \mathbb N$  we have
$$\pi_\Gamma(\mathcal A) \cap  \mathfrak F\o l_{\Gamma,Y}(n)\neq \emptyset.$$
The property (2-{$\mathcal A$}) ensures that ${\pi_\Gamma}_*(\chi_{E_i})=\chi_{\pi_{\Gamma}(E_i)}$, and therefore by Remark \ref{RF}, for all $n\in\mathbb N$
$$\{{\pi_\Gamma}_*(\chi_{E_1}), {\pi_\Gamma}_*(\chi_{E_2}),\ldots\}\cap \mathfrak Reit_{\Gamma,Y}(n) \neq\emptyset.$$
We now are in position to define the sought algorithm:\\
 for every $n\in \mathbb N$ we  run  the algorithm $\mathfrak K(n)$ used in the proof of Theorem \ref{imp}, simultaneously on the functions $\chi_{E_1}, \chi_{E_2},\ldots$ until one of the subroutines stops, providing a function $\chi$ such that ${{\pi_\Gamma}_*(\chi)\in \mathfrak Reit_{\Gamma,Y}(n)}$. Again, by the property (2-{$\mathcal A$}), the pushforward ${\pi_\Gamma}_*(\chi)$ is still a characteristic function and then by Remark \ref{RF}, the output $F:=Supp(\chi)$ (i.e. $\chi=\chi_F$) satisfies the required conditions.
\end{proof}

\section{Questions and final remarks}\label{QR}
The existence of a recursive universal bound for recursively (resp.\ finitely) presented amenable groups can be related to the arithmetic hierarchy of the property of being amenable. But there is no hope to establish, using our algorithm, if the bound is primitively recursive, since the stopping time depends on the bound itself.
\begin{que}
Is the class of recursively (finitely) presented amenable groups recursively enumerable?
\end{que}
\noindent For solvable groups the question is open (see \cite{miller}), even if in this case a universal bound for F\o lner functions of groups of this class is known \cite{SC}. In \cite{gri} there are some questions and remarks about decidability of amenability and bounds for F\o lner function in some subclasses of groups.
\vspace{6mm}

The Kharlampovich groups $G(M)$ have:
\begin{itemize}
\item unsolvable Word Problem \cite{H};
\item solvable generic Word Problem \cite{KMSS};
\item unsolvable strongly generic Word Problem \cite{GMO};
\item unsolvable generic Equality Problem (Corollary in the Introduction);
\item computable F\o lner sets \cite{CAV,Pre}.
\end{itemize}
Here a subset $ S \subset \mathbb F_{\mathsf X}$ is \emph{strongly generic} if $\frac{|S\cap B_n|}{|B_n|}\to 1$ exponentially fast, and a strongly generic problem is  solvable if it is solvable on a strongly generic set (for some generating set).\\
As an easy consequence, we deduce that solvability of generic WP does not imply solvability of generic EP.
\begin{que}
Does solvability of the strongly generic WP imply solvability of the (strongly) generic EP?
\end{que}
\noindent An answer to this question  would make clearer the relation between Theorem \ref{B} and the following.
\begin{ttt}[{\cite[Thm. 2.3]{GMO}}]
Let $G$ be a finitely presented amenable group with unsolvable word problem. Then for any choice of generators $W\to G$  the word problem in  $G$  is not solvable on any exponentially generic subset of $W$.
\end{ttt}

\noindent We can also measure genericity for the Equality Problem in $\mathbb F_{\mathsf X}\times \mathbb F_{\mathsf X}$ with a general subset, not necessarily of type $S\times S$, that is
$$T\subset \mathbb F_{\mathsf X}\times \mathbb F_{\mathsf X} \mbox{ is $(B_n\times B_n)$-\emph{generic} if } \frac{|T\cap (B_n\times B_n)|}{|B_n\times B_n|}\to 1.$$
With this weaker notion of genericity for the EP it is not clear if we can reach the analogous thesis of Theorem \ref{B}.

Finally, the last question that we asked in \cite{Pre}: --Does subrecursivity of the F\o lner function imply computability of F\o lner sets?-- can be replaced, in view of Theorem \ref{C}, by the following.
\begin{que}
Does there exist a recursively presented amenable group that has not computable F\o lner sets?
\end{que}

\bibliographystyle{amsxport}

\end{document}